\documentclass[11pt]{article}

\usepackage{graphicx, amsmath, amssymb, amsthm, color}
\usepackage[margin=2.5cm]{geometry}
\usepackage{mathtools}
\usepackage{float}
\usepackage{caption, keyval, subfig}
\usepackage{commath}              

\usepackage{commath}              
\usepackage{algorithmic}			
\usepackage{algorithm}
\usepackage{placeins}                 

\usepackage{authblk}

\usepackage{setspace}
\newtheorem{theorem}{Theorem}[section]                
\newtheorem{proposition}[theorem]{Proposition}       
\newtheorem{corollary}[theorem]{Corollary}
\newtheorem{lemma}[theorem]{Lemma}

\theoremstyle{definition}
\newtheorem{definition}[theorem]{Definition}
\newtheorem{remark}[theorem]{Remark}
\newtheorem{example}[theorem]{Example}

\begin{document}

\title{A global approach to the refinement of manifold data}

\author{Nira Dyn\thanks{niradyn@post.tau.ac.il} }
\author{Nir Sharon\thanks{Nir.Sharon@math.tau.ac.il}}
\affil{School of Mathematical sciences, Tel-Aviv University, Tel-Aviv, Israel}

\date{\vspace{-5ex}}
\maketitle

\begin{abstract} 
A refinement of manifold data is a computational process, which produces a denser set of discrete data from a given one. Such refinements are closely related to multiresolution representations of manifold data by pyramid transforms, and approximation of manifold-valued functions by repeated refinements schemes. Most refinement methods compute each refined element separately, independently of the computations of the other elements. Here we propose a global method which computes all the refined elements simultaneously, using geodesic averages. We analyse repeated refinements schemes based on this global approach, and derive conditions guaranteeing strong convergence.
\end{abstract}

\maketitle


\textbf{AMS(MOS) subject classification.} 65D99, 40A99, 58E10.



\section{Introduction}

In recent years many modern sensing devices produce data on manifolds or data that is modelled as points on a manifold. An example of such data is orientations of a rigid body as function of time, which can be regarded as data sampled from a function mapping a real interval to the Lie group of orthogonal matrices \cite{donoho}. The classical methods for the approximation of a function from its samples, such as polynomial or spline interpolation, are linear, and there is no guarantee that such approximations produce always manifold values, due to the non-linearity of manifolds. Therefore, alternative methods are required.

Contrary to the development of classical approximation methods and numerical analysis methods for real-valued functions, the development in the case of manifold-valued functions, which is rather recent, was mainly concerned in its first stages with advanced numerical and approximation processes. Examples of such processes are  geometric integration of ODE on manifolds (see e.g. \cite{iserles2000lie}), subdivision schemes on manifolds (see e.g. \cite{GrohsWallner2,ThomasLie1}) and wavelets-type approximation on manifolds (see e.g. \cite{grohs2012definability,donoho}). 

Subdivision schemes were created originally to design geometrical models \cite{ChaikinPaper,Lane_Reisenfeld_Original}. Later, they were recognized as methods for approximation \cite{dubucDeslauriers1,4Point}. The important advantage of these schemes is their simplicity and locality. They are defined by repeatedly refining sequences of points, applying in each refinement step simple and local arithmetic averaging. This enables the extension of subdivision schemes to more abstract settings, such as matrices \cite{NirUri} and sets \cite{dyn2001spline}.

For manifold valued data, Wallner and Dyn \cite{WallnerNiraProx} introduced the concept of adapting linear subdivision schemes to manifold data, and in particular for Lie group data. That paper initiated a new path of research on subdivision schemes for manifold data, e.g., \cite{NirUri,GrohsWallner2}. Adaptation of a linear subdivision scheme can be done in several ways, for example, by rewriting the refinement rules as repeated binary averages, and then replacing each binary average by a geodesic average, see e.g., \cite{NirUri,WallnerNiraProx}. 

Averages play a significant role in the methods for the adaptation of linear subdivision schemes to manifold data. A natural choice of an average of two points on a geodesically complete manifold is the midpoint of the geodesic curve between the two points. In some cases, the geodesic curve is known explicitly, e.g., \cite{faridi1987geodesics, genin2007geodesics_journal, RiemanSPD,marenich1997geodesics}, while in general it can be calculated numerically, e.g., \cite{crane2013geodesics,fucik1980nonlinear,kimmel1998computing,memoli2001fast}. 

The weighted geodesic average is induced by the geodesic curve, and acts as a generalization of the weighted arithmetic average $(1-t)a+tb$ in Euclidean spaces. For a weight $t \in [0,1]$, it is the point on the geodesic curve, connecting the two averaged points, which divides this curve segment in the ratio $\frac{t}{1-t}$.  Furthermore, on several manifolds, the geodesic average can also be extended to weights outside $[0,1]$, that is extrapolating the geodesic curve of two points beyond these points, e.g., \cite{UriNir}. The geodesic average is also well-defined on more general spaces known as geodesic metric spaces, e.g., \cite{bridson1999metric}. Thus, in such spaces our adaptation method is also valid.

We present here a method for the adaptation of linear subdivision schemes to manifold data based on the idea of replacing weighted arithmetic averages by weighted geodesic averages in a generalized Lane-Riesenfeld algorithm \cite{Lane_Reisenfeld_Original}. The refinement step in this proposed generalization consists of an elementary refinement of doubling the data, followed by several rounds of averaging. In each round of averaging the data is replaced by the same weighted average of all pairs of adjacent points in the data. Such an adaptation is discussed shortly in \cite{dyn2006three,WallnerNiraProx}. We term such a refinement step ``global refinement".

Many results, concerning the convergence and smoothness of adapted subdivision schemes, are presented in the literature of the past few years, e.g., \cite{GrohsWallner2,WallnerNiraProx,ThomasLie1}. Most of these results are based on proximity conditions. A proximity condition bounds the distance between the operation of an adapted refinement step to the operation of its linear counterpart in terms of the maximal distance between adjacent data points. Such proximity conditions hold, since a manifold is locally close to a Euclidean space. Thus, the convergence results are often valid only for ``dense enough data", which is, in general, a condition that is hard to quantify and depends on properties of the manifold (such as curvature).

Recently, a progress in the convergence analysis is established in several papers which address the question of convergence from any initial data. Such a result is presented in \cite{ebner2013convergence} for adapted subdivision schemes to data in Hadamard spaces. Results for data on the manifold of positive definite matrices are derived in \cite{NirUri}. For the case of interpolatory subdivision schemes there are also results for several different metric spaces e.g., \cite{UriNir,kels2013subdivision,wallnerconvergent}. 

Here we prove convergence from all initial data, of the above adapted generalized Lane-Riesenfeld algorithm, when the weighted average in each round corresponds to a weight in $[0,1]$, and give conditions for such convergence when some averages have weights outside $[0,1]$. In addition, we extend the above construction to a wider class of linear schemes, by introducing weighted trinary averages based on geodesic weighted averages, and give sufficient conditions for convergence from all initial data. In all these cases, and for manifolds with globally bounded curvature, the convergence guarantees that the limits are $C^1$, based on the proximity analysis in \cite{WallnerNiraProx}. 

Three important observations on our adaptation method: 
\begin{enumerate}
\item
It extends the class of linear schemes for which an adapted scheme is known to be convergent from all initial data.
\item
It is well-defined and convergent from all data in a wide class of geodesic metric spaces.
\item
It leads to computationally feasible subdivision schemes. 
\end{enumerate}

The convergence analysis introduced in this paper supplies a new tool for the analysis of linear schemes. In particular, this analysis guarantees the convergence of any linear scheme with a symbol which is a Hurwitz polynomial, up to multiplication by a monomial. The question whether this method can improve our ability to determine the convergence of linear subdivision schemes is beyond the scope of this paper and is still under investigation.

The paper is organized as follows. We start in Section \ref{sec:theoreticalBack} by providing a short survey of the required background, including a summary on the Lane-Riesenfeld algorithm and a short review on geodesics and manifolds. We conclude Section \ref{sec:theoreticalBack} with a short discussion on a sufficient condition for the convergence of adapted subdivision schemes. Section \ref{sec:global_refinements} introduces our generalization of the Lane-Riesenfeld algorithm. Then, we give conditions for the convergence of an adapted scheme based on this algorithm, from any initial manifold data, where the corresponding linear scheme has a factorizable symbol over the reals. In Section \ref{sec:global_refinements_complex_root} we further extend the algorithm to the adaptation of general linear schemes, and conclude the paper by the convergence analysis of these schemes.


\section{Preliminaries} \label{sec:theoreticalBack}

\subsection{Subdivision schemes and the Lane-Riesenfeld algorithm} \label{subsec:classical_subdivision_schemes}

Linear, univariate subdivision schemes are defined on numbers (the functional setting) , and are extended to vectors by operating on each component separately. In the functional setting, these schemes are approximation operators, when the data is sampled uniformly from a continuous function $f$. We denote the sampled data $f_i = f(ih)$, $i \in \mathbb{Z}$, $h>0$ by $\mathbf{f} = \{ f_i\}_{i \in \mathbb{Z}}$. Any subdivision scheme consists of refinement rules that map $\mathbf{f}$ to a new sequence $\mathcal{S}(\mathbf{f})$ associated with the values at $ih/2$, $i \in \mathbb{Z}$. 

Let us denote by $\mathcal{S}$ a refinement rule, defined by a finitely supported mask $\mathbf{a}:\mathbb{Z} \to \mathbb{R} $, as
\begin{equation} \label{eqn:classical_subdivision_scheme}
\mathcal{S}(\mathbf{f})_j = \sum_{i \in \mathbb{Z}} a_{j-2i} f_i .
\end{equation}
A (stationary) subdivision scheme with a refinement rule $\mathcal{S}$ is a repeated application of \eqref{eqn:classical_subdivision_scheme} and is also denoted by $\mathcal{S}$.

A subdivision $\mathcal{S}$ is termed convergent if the sequence of piecewise linear interpolants to the data $(i2^{-k},\mathcal{S}^k(\mathbf{f})_i )$ converges uniformly (see e.g. \cite{NiraScheme}). By definition, the limit is a continuous function.

The Lane-Reisenfeld (L-R) algorithm is a classical algorithm, which executes the refinement rules of a B-spline subdivision scheme \cite{Lane_Reisenfeld_Original}. This algorithm replaces each step of refinement by an elementary refinement (doubling all the data points) followed by several stages of averaging. In each stage of averaging, the data points are replaced by the mid-points of all pairs of consecutive data points. As a result, the refinement is done simultaneously to all data points. We term this refinement a global refinement, in contrary to the direct evaluation of \eqref{eqn:classical_subdivision_scheme}, where each refined point is calculated independently of the other refined points. The refinement step of the L-R algorithm is presented in Algorithm \ref{alg:Lane-Reisenfeld}. 
 
 \begin{algorithm}[ht]
\caption{The refinement step of the Lane-Reisenfeld algorithm }
\label{alg:Lane-Reisenfeld}
\begin{algorithmic}[1]
\REQUIRE  The data to be refined $ \mathbf{f} = \{f_i\}_{i \in \mathbb{Z}}$. The degree of the B-spline $m$.
\ENSURE The refined data $\mathcal{S} \left( \mathbf{f} \right) $.
\STATE  $q_{2i,0} \gets p_{i}$    \label{LRalg:init1}
\STATE  $q_{2i+1,0} \gets p_{i}$  \label{LRalg:init2}
\FOR{$j=1$ \TO $m$}    \label{LRalg:factor_loop}
\FOR{$i \in \mathbb{Z}$}  \label{LRalg:main_loop}
\STATE  $q_{i,j} \gets \frac{1}{2}(q_{i,j-1}+q_{i+1,j-1}) $ \label{LRalg:average_loop}
\ENDFOR     \label{LRalg:data_const2}
\ENDFOR
\RETURN $\{q_{i,m} \}_{i \in \mathbb{Z}}$
\end{algorithmic}
\end{algorithm}

An important tool in the analysis of convergence and smoothness of subdivision schemes is the symbol, defined as the $z$-transform of the mask $\mathbf{a}$, that is $a(z) =\sum_{i \in \mathbb{Z}} a_iz^i$. For example, the symbol of the B-spline subdivision scheme of degree $m$ is $a(z) = (1+z)^{m+1}/2^m$. A necessary condition for convergence is $a(1)=2$ and $a(-1)=0$ implying that the subdivision scheme is invariant to a translation of the data \cite[Proposition 2.1]{NiraScheme}. With the symbol $a(z)$ the refinement rules \eqref{eqn:classical_subdivision_scheme} can be written algebraically as
\begin{equation} \label{eqn:algebraic_refinement}
\sum_{j \in \mathbb{Z}} \mathcal{S}(\mathbf{f})_j z^j = a(z) \sum_{j \in \mathbb{Z}}  f_j z^{2j} , 
\end{equation}
where the equality is in the sense of equal coefficients corresponding to the same power of $z$. The L-R algorithm is an interpretation of \eqref{eqn:algebraic_refinement} with the symbols of the B-spline subdivision schemes. For explanation see Section \ref{subsec:gloabl_refinement_1} and in particular \eqref{eqn:symbol_as_aveg}.

Over the years, several generalizations of the L-R algorithm have been proposed. In \cite{cashman2013generalized} any step of the subdivision consists of a refinement step of a fixed converging subdivision scheme, followed by a fixed number of ``smoothing rounds" based on another subdivision scheme (e.g., applying the insertion rule of an interpolatory scheme to each point). In \cite{dyn2011convergence,Goldman} non-linear averages of numbers replace the arithmetic (linear) averages. A generalization based on a geodesic average goes back to \cite{noakes1998nonlinear,noakes1999accelerations} where a corner cutting subdivision scheme based on geodesic averages is presented and analysed. In \cite{dyn2001spline} the L-R algorithm is adapted to compact sets based on the metric average which is a geodesic average in the metric space of compact sets with the Hausdorff metric.

In this paper we discuss the adaptation of subdivision schemes from numbers to manifold data. To distinguish between sequences of numbers (or vectors) to sequences on a manifold, we denote by $\mathbf{f} = \{ f_i \}_{i \in \mathbb{Z}}$ and $\mathbf{p} = \{ p_i \}_{i \in \mathbb{Z}}$ a sequence of Euclidean data and manifold, respectively. 

\subsection{On manifolds and geodesics} \label{subsec:manifolds_and_geodesics}

A \textit{geodesic} (or a geodesic curve) is a fundamental notion in differential geometry. This notion is an extension of the shortest arc on a surface, joining two arbitrary points $p_1$ and $p_2$ on the surface. On a plane, the geodesic is simply the line segment connecting $p_1$ and $p_2$, described by 
\begin{equation} \label{eqn:arithmetic_mean}
 (1-t)p_1+tp_2 , \quad t \in [0,1] . 
\end{equation}
This line can be also characterized by its zero curvature and its endpoints. For a manifold, this property is generalized by having zero geodesic curvature (or constant velocity derived from the first fundamental form). In Riemannian manifolds, the geodesic curve is defined as the solution to the geodesic Euler-Lagrange equations. It turns out that any shortest path between two points is a geodesic curve. 

In connected Riemannian manifolds, the Hopf-Rinow theorem guarantees that geodesic curves connecting any two points are globally well defined and smooth, see e.g., \cite{do1992riemannian}. Such manifolds are also known as geodesically complete or simply complete Riemannian manifolds. For such manifolds, one can derive the uniqueness of the geodesic curve connecting any two points, in case one point is outside the cut locus of the other. Henceforth, we will use the term geodesic curve for such shortest path curves.

The geodesic curve is of great importance in our adaptation procedures. A natural question is its availability in different manifolds. Indeed, in many cases, the geodesic curve is known explicitly. Here are several examples: on a sphere (e.g., \cite{faridi1987geodesics}), on an ellipsoid (e.g., \cite{genin2007geodesics_journal}), on the cone of positive definite matrices (e.g., \cite{RiemanSPD}), in the Lie group of orthogonal matrices of the same determinant (e.g., \cite[Chapter 3]{StillwelLie}), in the Heisenberg groups (e.g., \cite{marenich1997geodesics}). Alternatively, geodesics can be calculated numerically. This can be done by directly solving the Euler-Lagrange equations (e.g., \cite{fucik1980nonlinear}), by fast marching methods (e.g., \cite{kimmel1998computing}), by exploiting heat kernels based methods (e.g., \cite{crane2013geodesics}), or other hyper-surfaces techniques (e.g., \cite{memoli2001fast}), just to name a few.

An important property of the geodesic curve is the \textit{metric property}. Let $\mathcal{M}$ be a complete Riemannian manifold with associated metric $d$. Then, for any $p_1,p_2 \in \mathcal{M}$ the geodesic curve connecting $p_1$ and $p_2$, that is $M_t(p_1,p_2)$, $t \in [0,1]$ with $M_0(p_1,p_2)=p_1$ and $M_1(p_1,p_2)=p_2$, satisfies 
\begin{equation} \label{eqn:metric_property}
d(M_t(p_1,p_2),p_2) =  (1-t)d(p_1,p_2) , \quad t \in [0,1] . 
\end{equation}
Since $d$ is a metric, we also have the compliment formula $d(p_1,M_t(p_1,p_2)) =  td(p_1,p_2)$. In this paper, we consider data $\mathbf{p}$ such that the geodesic curve between any two adjacent data points in $\mathbf{p} $ is well-defined, and term such data ``admissible". Then, the geodesic curve $M_t$ is used as a weighted $t$ mean, that is the manifold analogue of the arithmetic mean \eqref{eqn:arithmetic_mean}. In some cases, we may need $M_t$ to be defined for values of $t$ outside $[0,1]$, but close to it. Therefore, we must assume that the geodesic curve is well-defined for these ``extrapolation" values. In these cases the metric property \eqref{eqn:metric_property} is modified, replacing $1-t$ by $\abs{1-t}$.

There are some non-linear spaces, other than Riemannian manifolds, where the geodesic curve connecting any two points is unique. These are the geodesic metric spaces, see e.g., \cite{bridson1999metric}. In such spaces, the differential structure is missing and a geodesic curve is defined as the path satisfying \eqref{eqn:metric_property}. Clearly, this definition agrees with the geodesic curve on Riemannian manifolds. Note that, in general, we do not need the uniqueness of the geodesic curve, but a canonical way to choose it, see e.g., \cite{dyn2001spline}.

\subsection{Sufficient conditions for convergence of manifold-valued subdivision schemes}

The convergence of manifold-valued subdivision schemes can be defined intrinsically. For that, we defined for any data sequence $\mathbf{p}$, a piecewise geodesic interpolant $I(\mathbf{p})$, connecting any pair of consecutive points in $\mathbf{p}$ by their geodesic curve. The manifold-valued subdivision scheme $\widetilde{ \mathcal{S}}$ is convergent, if the sequence $I(\widetilde{ \mathcal{S}}^k(\mathbf{p})   )$, $k \in \mathbb{Z}_+$ converges uniformly relative to the metric of the manifold (see \cite{NiraManofold1}).

The analysis of adapted subdivision schemes in many papers is based on the method of proximity, introduced in \cite{WallnerNiraProx}. This analysis uses conditions that indicate the proximity of the adapted refinement rule $\widetilde{ \mathcal{S}}$ to its corresponding linear refinement rule $ \mathcal{S}$. The simplest proximity condition is
\begin{equation} \label{eqn:proximity_condition}
d \left( \mathcal{S}(\mathbf{p}),\widetilde{ \mathcal{S}}(\mathbf{p}) \right)  \le c  \left( \delta(\mathbf{p}) \right)^2 , \quad \delta(\mathbf{p}) = \sup_{i \in \mathbb{Z}} d(p_i,p_{i+1}) , \quad c \in \mathbb{R}_+ .
\end{equation}
In \cite{WallnerNiraProx} it is proved that if $ \mathcal{S}$ is a refinement rule of a convergent scheme that generates $C^1$ limits, then condition \eqref{eqn:proximity_condition} implies (with additional mild assumptions on the refinement rule $\mathcal{S}$) that for $\delta(\mathbf{p})$ small enough, the adapted subdivision scheme $\widetilde{ \mathcal{S}}$, applied to the initial data $\mathbf{p}$, converges to a $C^1$ limit.

The weakness of the proximity method is that convergence is only guaranteed for ``close enough" data points. This requirement is typically not easy to quantify and it depends on the manifold and its curvature. 

For a linear subdivision schemes a contractivity factor $\mu$, namely
\begin{equation} \label{eqn:contractivity_factor_general}
 \delta\left( \mathcal{S}(\mathbf{p}) \right)  \le \mu \delta(\mathbf{p})   , \quad \mu \in (0,1), 
\end{equation}
implies the convergence of the scheme from any initial data, see e.g. \cite{NiraScheme}. 

For non-linear subdivision schemes, and in particular for schemes adapted to manifold data, contractivity is not sufficient for convergence, and an additional condition is required, see \cite{NiraManofold1}.

\begin{definition}[Displacement-safe] 
Let $\widetilde{\mathcal{S}}$ be a subdivision scheme adapted to manifold data. We say that $\widetilde{\mathcal{S}}$ is ``displacement-safe" if
\begin{equation} \label{eqn:contractivity_extra_condition}
d(\widetilde{\mathcal{S}}(\mathbf{p})_{2i},(\mathbf{p})_i )  \le C \delta(\mathbf{p} ) ,\quad i \in \mathbb{Z} .
\end{equation}
for any sequence of manifold data $\mathbf{p}$, where $C$ is a constant independent of $\mathbf{p}$.
\end{definition}

In \cite{NiraManofold1}, it is proved that
\begin{theorem} \label{thm:contraction_convergence}
Let $\widetilde{\mathcal{S}}$ be a displacement-safe subdivision scheme for manifold data with a contractivity factor $\mu<1$. Then, $\widetilde{\mathcal{S}}$ is convergent for any input manifold data.
\end{theorem}

\begin{remark} \label{rem:two_covergence_remarks}
Two concluding remarks:
\begin{enumerate}
\item 
Note that interpolatory schemes satisfy \eqref{eqn:contractivity_extra_condition} with $C=0$ by definition and thus are displacement-safe.
\item  \label{rem:TCR_part2}
In \cite{WallnerNiraProx} it is proved that any adaptation of \eqref{eqn:classical_subdivision_scheme} based on repeated geodesic averages satisfies \eqref{eqn:proximity_condition}, under mild assumptions on the manifold, such as manifolds with globally bounded curvature. This observation implies that for $\mathbf{p}$ with $\delta(\mathbf{p})<1$, \eqref{eqn:contractivity_extra_condition} is also satisfied. Thus, for such schemes, it is enough to show that the scheme has a contractivity factor to obtain convergence for any initial data and to conclude that the limit is $C^1$.
\end{enumerate}
\end{remark}


\section{Adaptation of generalized L-R algorithms} \label{sec:global_refinements}

We present an adaptation method of generalized L-R algorithms, based on geodesic averages. This method is already introduced in \cite{dyn2006three,WallnerNiraProx}. Nevertheless, the convergence result stated there is the one that follows from proximity conditions, which applies only for $\delta(\mathbf{p})$ small enough. First, we discuss in detail our adaptation and then analyze the resulting schemes, charactering classes of schemes for which convergence from any initial data is guaranteed. 

\subsection{The algorithm of global refinement} \label{subsec:gloabl_refinement_1}

Consider a linear subdivision scheme $\mathcal{S}$ of the form \eqref{eqn:classical_subdivision_scheme}, with a symbol $a(z)=\sum_{j \in \mathbb{Z}} a_jz^j$. The factorization of the symbol plays an important role in the analysis of convergence and smoothness of linear subdivision schemes \cite{NiraScheme}, and is also significant in our adaptation.

We start with a class of convergent linear subdivision schemes having symbols which can be factorized into real linear factors. Recall that a necessary condition for convergence is that $a(-1)=0$ and $a(1)=2$ \cite[Proposition 2.1]{NiraScheme}. Thus, we can write
\begin{equation} \label{eqn:symbol_linear_factorization}
a(z) = z^{-s} (1+z) \frac{1+\alpha_1z}{1+\alpha_1} \cdots \frac{1+\alpha_mz}{1+\alpha_m} , 
\end{equation}
where $-\alpha_1^{-1},\ldots -\alpha_m^{-1}$ are the nonzero roots of the symbol and $s$ is an integer. Note that $1$ cannot be a root of a symbol since $a(1) = 2$. Thus, $\alpha_j  \neq -1$. $j=1,\ldots,m$ and \eqref{eqn:symbol_linear_factorization} is well-defined. We further define $\alpha_1$ to be the minimizer of
\begin{equation} \label{eqn:alpha_1_minimization}
 \max(\frac{1}{1+\alpha_j},\frac{\alpha_j}{1+\alpha_j})   ,   
\end{equation}
among $\alpha_1,\ldots,\alpha_m$. The reason will become clear later.

The relation between the factorization \eqref{eqn:symbol_linear_factorization} and the global refinement is based on \eqref{eqn:algebraic_refinement}. For the symbol \eqref{eqn:symbol_linear_factorization} we get from \eqref{eqn:algebraic_refinement} that the linear scheme can be interpreted as
\begin{equation} \label{eqn:symbol_as_aveg}
\begin{array}{r@{}l}
 \sum_{j \in \mathbb{Z}} \mathcal{S}(\mathbf{f})_j z^j &{}= z^{-s} \prod_{i=1}^m \frac{1+\alpha_iz}{1+\alpha_i} \left( (1+z) \sum_{j \in \mathbb{Z}}  f_j z^{2j} \right) \\
&{} = z^{-s} \left(\prod_{i=2}^m  \frac{1+\alpha_iz}{1+\alpha_i} \right) \left( \frac{1+\alpha_1z}{1+\alpha_1} \right)  \sum_{j \in \mathbb{Z}} \left( f_j z^{2j} + f_j z^{2j+1} \right) \\
&{} = z^{-s+1} \prod_{i=2}^m \frac{1+\alpha_iz}{1+\alpha_i}  \sum_{j \in \mathbb{Z}} \left( ( \frac{f_j+\alpha_1 f_{j-1}}{1+\alpha_1}) z^{2j-1} + f_j z^{2j} \right) .
\end{array}
\end{equation}
By this interpretation, the factor $1+z$ indicates an initial elementary refinement step in which the data is duplicated. Then, each of the factors $\frac{1+\alpha_jz}{1+\alpha_j} $, $j=1,\ldots,m$ implies a step of averaging, in which the current data is replaced by the weighted averages with weights $\frac{1}{1+\alpha_j},\frac{\alpha_j}{1+\alpha_j} $ on its pairs of adjacent points. A zero root of the symbol merely changes the value of $s$. This value determines the shift of indices required to be applied, at the end of each refinement step. Note that for $\alpha_i =1$, $i=1,\ldots,m$, this interpretation becomes the L-R algorithm. Thus, we consider the global refinement step corresponding to \eqref{eqn:symbol_linear_factorization} a generalized L-R algorithm.

The adaptation of the global refinement, based on geodesic averages, is summarized in Algorithm \ref{alg:Global_refinement}.

\begin{algorithm}[ht]
\caption{Global refinement step}
\label{alg:Global_refinement}
\begin{algorithmic}[1]
\REQUIRE  The values $s$ and $\alpha_1,\ldots,\alpha_m$ of the symbol \eqref{eqn:symbol_linear_factorization}. \\ The data to be refined by $\mathcal{S}$, $ \mathbf{p} = \{p_i\}_{i \in \mathbb{Z}}$.
\ENSURE The refined data $\mathcal{S} \left( \mathbf{p} \right) $.
\STATE  $q_{2i,0} \gets p_{i}$    \label{alg:init1}
\STATE  $q_{2i+1,0} \gets p_{i}$  \label{alg:init2}
\FOR[Go over each term in the factorization of the symbol]{$j=1$ \TO $m$}    \label{alg:factor_loop}
\FOR{$i \in \mathbb{Z}$}  \label{alg:main_loop}
\STATE  $q_{i,j} \gets M_\frac{\alpha_j}{1+\alpha_j}(q_{i,j-1},q_{i+1,j-1}) $ \label{alg:average_loop}
\ENDFOR     \label{alg:data_const2}
\ENDFOR
\FOR[A final shifting]{$i \in \mathbb{Z}$} 
\STATE   $\mathcal{S} \left( \mathbf{p} \right)_{i-s+1}  \gets q_{i,m} $. \label{alg:shifting_step}
\ENDFOR     
\RETURN $\mathcal{S} \left( \mathbf{p} \right)$
\end{algorithmic}
\end{algorithm}

Note that for data sampled from a geodesic curve, all points generated by Algorithm \ref{alg:Global_refinement}, are on this geodesic curve.

\subsection{Analysis of schemes corresponding to factorizable symbols over the reals} \label{subsec:analysis_linear factor}

For our first result, we restrict the discussion to the case where the symbol \eqref{eqn:symbol_linear_factorization} has a full set of real negative roots, namely $\alpha_i>0$, $i=1,\ldots,m$. 
\begin{theorem} \label{thm:global_contraction1}
Let $\mathcal{S}$ be a linear subdivision scheme with the symbol \eqref{eqn:symbol_linear_factorization}, such that $\alpha_j>0$, $j=1,\ldots,m$. Then, the adapted scheme based on the global refinement step of Algorithm \ref{alg:Global_refinement} has a contractivity factor $ \mu =  \max\{ \frac{1}{1+\alpha_1},\frac{\alpha_1}{1+\alpha_1}  \} $. 
\end{theorem}
\begin{proof}
Following Algorithm \ref{alg:Global_refinement} we get that after the initial stage of Line \ref{alg:init1} and Line \ref{alg:init2} we have that
\[  d(q_{2i,0},q_{2i+1,0}) = 0, \qquad d(q_{2i-1,0},q_{2i,0}) \le \delta(\mathbf{p}) , \quad i \in \mathbb{Z} . \]
After the first iteration of the loop of Line \ref{alg:factor_loop} we have (see \eqref{eqn:symbol_as_aveg})
\[  q_{2i,1} = q_{2i,0} , \qquad q_{2i+1,1} = M_{\frac{\alpha_1}{1+\alpha_1}}(q_{2i+1,0},q_{2i+2,0}) , \quad i \in \mathbb{Z} . \] 
By the metric property \eqref{eqn:metric_property},
\[  d(q_{2i,1},q_{2i+1,1}) = \frac{1}{1+\alpha_1} \delta(\mathbf{p}), \qquad d(q_{2i-1,0},q_{2i,0}) \le \frac{\alpha_1}{1+\alpha_1}\delta(\mathbf{p}) , \quad i \in \mathbb{Z} . \]
Thus, for $\mathbf{q}^{[1]} = \{ q_{i,1} \}_{i \in \mathbb{Z}} $, $\delta(\mathbf{q}^{[1]} ) \le \mu \delta(\mathbf{p})$ with $ \mu =  \max\{ \frac{1}{1+\alpha_1},\frac{\alpha_1}{1+\alpha_1}  \} $. The next iterations, $j=2,\ldots,m$, retain the maximal bound of $\mu \delta(\mathbf{p}) $, since for $j>1$ 
\[  d(q_{i,j},q_{i+1,j} ) \le d(q_{i,j}, q_{i+1,j-1} ) + d(q_{i+1,j-1}.q_{i+1, j} ) \le \frac{\alpha_j}{1+\alpha_j} \mu \delta(\mathbf{p}) + \frac{1}{1+\alpha_j} \mu \delta(\mathbf{p}) = \mu \delta(\mathbf{p}) .\]
\end{proof}

Note that the contractivity factor of Theorem \ref{thm:global_contraction1} satisfies $\mu \ge \frac{1}{2}$ since $\frac{\alpha_1}{1+\alpha_1},\frac{1}{1+\alpha_1} \in (0,1)$ and $\frac{\alpha_1}{1+\alpha_1}+\frac{1}{1+\alpha_1} = 1$, with $\mu=\frac{1}{2}$ for $\alpha_1=1$.

The L-R algorithm satisfies the conditions of Theorem \ref{thm:global_contraction1}. Indeed, this theorem is a generalization of a similar result in \cite[Lemma 4.1]{dyn2001spline} for the adapted L-R algorithm to compact sets.

Next, we show that the adapted subdivision schemes corresponding to symbols having a full set of real negative roots, are displacement-safe.
\begin{theorem} \label{thm:DS_condition_negative_roots}
Let $\mathcal{S}$ be as in Theorem \ref{thm:global_contraction1}. Denote by $\widetilde{\mathcal{S}}$ the adapted scheme based on the global refinement of Algorithm \ref{alg:Global_refinement}. Then, $\widetilde{\mathcal{S}}$ is displacement-safe.
\end{theorem}
\begin{proof}
The proof shows by induction that $d( \widetilde{\mathcal{S}}(\mathbf{p})_{2i},p_i   ) \le K_m \delta(\mathbf{p})$, $i \in \mathbb{Z}$. Denote by $\mathcal{S}_j$ the linear subdivision scheme with a symbol obtained from the symbol of $\mathcal{S}$ by retaining the first $j$ factors, $1 \le j \le m$, so that the adapted scheme of $\mathcal{S}_j$, $\widetilde{\mathcal{S}}_j$, uses only $j$ iterations of the loop of Line \ref{alg:factor_loop} in Algorithm \ref{alg:Global_refinement}. Obviously $\mathcal{S} = \mathcal{S}_m$. We use induction on $j$. For $j=1$, after the initial steps of Lines \ref{alg:init1} and \ref{alg:init2}, Algorithm \ref{alg:Global_refinement} inserts new points on the geodesic curves, connecting adjacent data points. Therefore, it is clear that we have $ d( \widetilde{\mathcal{S}}_1(\mathbf{p})_{2i},\mathbf{p}_i   ) \le \delta(\mathbf{p})$, namely we get the constant $K_1 =1 $ for the case $j=1$. The induction step assumes  
\[ d( \widetilde{\mathcal{S}}_j(\mathbf{p})_{2i},p_i   ) \le K_j \delta(\mathbf{p})  , \quad  i \in \mathbb{Z} , \]
for a given $j$, $1 \le j < m-1$ with a constant $K_j$, which depends on $j$ and is independent of $\mathbf{p}$. Then, using the triangle inequality we get
\[ d( \widetilde{\mathcal{S}}_{j+1}(\mathbf{p})_{2i},p_i   ) \le d( \widetilde{\mathcal{S}}_{j+1}(\mathbf{p})_{2i},\widetilde{\mathcal{S}}_{j}(\mathbf{p})_{2i}   ) + d( \widetilde{\mathcal{S}}_{j}(\mathbf{p})_{2i},p_i   ) .\]
While by the metric property \eqref{eqn:metric_property} (see Line \ref{alg:average_loop} in Algorithn \ref{alg:Global_refinement})
\begin{equation} \label{eqn:bound_between_levels}
 d( \widetilde{\mathcal{S}}_{j+1}(\mathbf{p})_{2i},\widetilde{\mathcal{S}}_{j}(\mathbf{p})_{2i}   ) \le   \delta(\widetilde{\mathcal{S}}_{j}(\mathbf{p})   ) . 
\end{equation}
Since Theorem \ref{thm:global_contraction1} implies that
\begin{equation} \label{eqn:common_mu_factor}
 \delta(\widetilde{\mathcal{S}}_{j}(\mathbf{p})   )\le \mu \delta(\mathbf{p}) , \quad \mu =  \max\{ \frac{1}{1+\alpha_1},\frac{\alpha_1}{1+\alpha_1}  \}  , 
\end{equation}
we can choose $K_{j+1} = \mu+K_j $ and the proof follows. The shift, defined by $s$ in \eqref{eqn:symbol_linear_factorization} and done in Line \ref{alg:shifting_step} of Algorithm \ref{alg:Global_refinement}, does not affect the above bound, since $s$ is the same for all $\mathcal{S}_j$. 
\end{proof}

We conclude
\begin{corollary}
Let $\mathcal{S}$ be a linear subdivision scheme with the symbol \eqref{eqn:symbol_linear_factorization}, such that $\alpha_j>0$, $j=1,\ldots,m$. Then, the adapted scheme based on the global refinement of Algorithm \ref{alg:Global_refinement} converges for all admissible input data on the manifold.
\end{corollary}

The second case analyzed here corresponds to symbols of the form \eqref{eqn:symbol_linear_factorization} with several positive roots. Positive roots mean negative weights in the averages, namely extrapolating averages in Line \ref{alg:average_loop} of Algorithm \ref{alg:Global_refinement}. 

\begin{theorem} \label{thm:positive_rots_convergence} 
Let $\mathcal{S}$ be a linear convergent subdivision scheme with symbol $a(z)$ of the form \eqref{eqn:symbol_linear_factorization}, such that $a(z)$ has at least one negative root in addition to the root $-1$. Define
\[ \mu_1 = \min_{\substack{\alpha_i>0 \\ 
     i \in \{1,\ldots,m \} } }  \max\{ \frac{1}{1+\alpha_i},\frac{\alpha_i}{1+\alpha_i}  \} , \]
and renumerate the factors in \eqref{eqn:symbol_linear_factorization} such that $\mu_1$ is attained at $\alpha_1$. If
\begin{equation} \label{eqn:positive_roots_contractive_factor}
\mu = \mu_1 \prod_{i=2}^m \xi(\alpha_i) <1 , 
\end{equation}
where 
\[ \xi(\alpha) = 
\begin{cases} 
1,                                                     & 0<\alpha       , \\
1+2 \abs{\frac{\alpha}{1+\alpha}},  & -1<\alpha<0 , \\
1+2 \abs{\frac{1}{1+\alpha}},          & \alpha<-1      , 
\end{cases}
\]
then the adapted scheme based on global refinement has a contractivity factor $\mu$, and it converges from any admissible initial data on the manifold.
\end{theorem}
\begin{proof}
The proof basically modifies the proofs of Theorem \ref{thm:global_contraction1} and Theorem \ref{thm:DS_condition_negative_roots}. By assumption the set 
$\left\lbrace \alpha_i>0  \colon i \in \{1,\ldots,m \} \right\rbrace $ is not empty, and therefore $\mu_1 <1$. Similarly to the proof of Theorem \ref{thm:global_contraction1} the application of an averaging step in Line \ref{alg:average_loop} of Algorithm \ref{alg:Global_refinement}, corresponding to $\alpha_i>0$, does not expand the bound on the distances between consecutive points in the data. On the other hand, an averaging step corresponding to $\alpha_i<0$ expands the bound.

To obtain the expanding factor note that after the $j$-th step in Line \ref{alg:average_loop} of Algorithm \ref{alg:Global_refinement} we can bound the distance between consecutive points by
\begin{equation} \label{eqn:bound_on_line_5_positive_roots}
d(q_{i,j},q_{i+1,j}) \le d(q_{i,j},q_{i,j-1})+d(q_{i,j-1},q_{i+1,j-1})+d(q_{i+1,j-1},q_{i+1,j})   .
\end{equation}
Defining $\mu_j = \mu_1 \prod_{i=2}^j \xi(\alpha_i)$, $j=2,\ldots,m$, we obtain from \eqref{eqn:bound_on_line_5_positive_roots} 
\begin{equation} \label{eqn:positive_root_main_contractive}
 d(q_{i,j},q_{i+1,j}) \le \xi(\alpha_j) \mu_{j-1} \delta(\mathbf{p}) .  
\end{equation}
This together with assumption \eqref{eqn:positive_roots_contractive_factor} shows that $\mu = \mu_m$ is a contractivity factor of the adapted scheme.

To complete the convergence proof, we observe that since $\mu_1 \ge \frac{1}{2}$, assumption \eqref{eqn:positive_roots_contractive_factor} implies that $\xi(\alpha_i)<2$, $i=1,\ldots,m$. Modifying the proof of Theorem \ref{thm:DS_condition_negative_roots}, we get in its notation that \eqref{eqn:bound_between_levels} is replaced by
\[  d( \widetilde{\mathcal{S}}_{j+1}(\mathbf{p})_{2i},\widetilde{\mathcal{S}}_{j}(\mathbf{p})_{2i}   ) \le  2 \delta(\widetilde{\mathcal{S}}_{j}(\mathbf{p})   ) .  \] 
Using the same inductive argument, and the bound \eqref{eqn:positive_root_main_contractive}, we get
\begin{eqnarray*}
d(\widetilde{\mathcal{S}}_{j+1}(\mathbf{p})_{2i},p_i)   &\le&  d(\widetilde{\mathcal{S}}_{j+1}(\mathbf{p})_{2i},\widetilde{\mathcal{S}}_{j}(\mathbf{p})_{2i}) +d(\widetilde{\mathcal{S}}_{j}(\mathbf{p})_{2i},p_i)   \\ 
&\le & 2\delta(\widetilde{\mathcal{S}}_{j}(\mathbf{p}))+K_j  \delta(\mathbf{p}) \le (2\mu_{j} +K_j)  \delta(\mathbf{p})  .
\end{eqnarray*}
Thus, in this case $K_{j+1} = 2\mu_{j} +K_j $. By \eqref{eqn:positive_roots_contractive_factor} $\mu_{j} \le \mu <1$, and since $\alpha_1>0$ implies $K_1=1$, we finally arrive at $K_m = 1+2m$.

We conclude that the adapted scheme obtained from $\mathcal{S}$ by global refinement is displacement-safe and has a contractivity factor $\mu$ given in \eqref{eqn:positive_roots_contractive_factor}. Therefore, it converges by Theorem \ref{thm:contraction_convergence}.
\end{proof}

\begin{remark} \label{rem:inapp_for_interpolation}
Two remarks for section \ref{subsec:analysis_linear factor}:
\begin{enumerate}
\item
As is proved in Theorems \ref{thm:global_contraction1} and \ref{thm:DS_condition_negative_roots} the adaptation of Algorithm \ref{alg:Global_refinement} leads to converging subdivision schemes when applied to linear subdivision schemes with positive mask coefficients, such that their symbols have a full set of negative roots. Theorem \ref{thm:positive_rots_convergence} extends the convergence to schemes with symbols having few positive roots in addition to at least two negative ones, which may correspond to masks with some negative coefficients. 
\item
Negative coefficients necessarily appear in the masks of smooth interpolatory schemes. However, the adaptation based on global refinement is inappropriate for interpolatory subdivision schemes, since the adapted schemes are not interpolatory any more. The commutativity of multiplication of numbers guarantees that for numbers the local refinement and the global refinement coincide.
\end{enumerate}
\end{remark}

In the next section we show that the global refinement can be interpreted as local refinements, based on a ``pyramid averaging".

\subsection{interpretation of the global refinement as local refinement} \label{subsec:local_interpretation}

Most known adaptation methods of convergent linear subdivision schemes to manifold data are based on first rewriting the average \eqref{eqn:classical_subdivision_scheme} in terms of repeated binary averages, and then replacing the linear averages by some manifold averages, see e.g. \cite{GrohsWallner2,WallnerNiraProx,ThomasLie1}. We term the so obtained refinement rules ``local refinement".

Next we show that global refinement can be interpreted as local refinement based on geodesic averages. This observation together with \ref{rem:TCR_part2} of Remark \ref{rem:two_covergence_remarks} leads to the conclusion that the convergence of schemes adapted by global refinement guarantees $C^1$ limits.

We now describe how the global refinement can be interpreted as local refinement. For $i$ even, $\mathcal{S} \left( \mathbf{p} \right)_i$ in Algorithm \ref{alg:Global_refinement} can be calculated by a series of repeated averaging operating on $p_i,p_{i+1},\ldots,p_{i+\lfloor \frac{m}{2} \rfloor }$. First we replace $p_{\ell}$ by $M_0(p_{\ell},p_{\ell+1}),M_{\frac{\alpha_1}{\alpha_1+1}} (p_{\ell},p_{\ell+1})$, $\ell = i,\ldots ,i+\lfloor \frac{m}{2} \rfloor$. We take from this sequence the first $m$ points, to form the initial level for a ``pyramid averaging" of $m-1$ levels. In the $j$-th level of the pyramid averaging any pair of adjacent points is replaced by its geodesic average with weight $\frac{\alpha_{j+1}}{\alpha_{j+1}+1}$, $j=1,\ldots,m-1$. Thus at the $j$-th level there are $m-j$ points. $\mathcal{S} \left( \mathbf{p} \right)_i$ is the only value obtained at level $m-1$ of the pyramid averaging.

For $i$ odd, $\mathcal{S} \left( \mathbf{p} \right)_i$ in Algorithm \ref{alg:Global_refinement} can be calculated similarly, starting the same pyramid averaging from a different sequence. This sequence is obtained from $p_i,p_{i+1},\ldots,p_{i+\lceil \frac{m}{2} \rceil }$ by first replacing $p_{\ell}$ by $M_{\frac{\alpha_1}{\alpha_1+1}} (p_{\ell},p_{\ell+1}), M_1(p_{\ell},p_{\ell+1})$, $\ell = i,\ldots ,i+\lceil \frac{m}{2} \rceil -1$ and then taking the first $m$ points. For illustrations and explanation of the pyramid averaging notion see \cite{schaefer2009non}.

The global refinement calculates only once each geodesic averages of adjacent points in the data, while the same average appears in the calculation of several points by local refinement. Thus, the global refinement is more efficient in terms of computational operations as compared to its local refinement interpretation. Note that it is possible to define a scheme adapted by local refinement which uses the same number of geodesic averages as the global refinement \cite{NiraManofold1}.


\section{Adaptation based on global refinement -- the general case} \label{sec:global_refinements_complex_root}

We extend the global refinement algorithm to converging linear schemes with general symbols. Then, instead of \eqref{eqn:symbol_linear_factorization} such symbols, which are real polynomials, can be factorized into $m_1$ real linear factors (in addition to $1+z$) and $m_2$ quadratic real factors, with $m_1+2m_2=m$. Any complex root of the symbol corresponds to a real quadratic irreducible factor over the reals of the form
\begin{equation} \label{eqn:irreducible_factor}
 \frac{1+\alpha z}{1+\alpha} \cdot \frac{1+\overline{\alpha }z}{1+\overline{\alpha}} =\frac{1+2\operatorname{Re}(\alpha)z+\abs{\alpha}^2 z^2 }{1+2\operatorname{Re}(\alpha)+\abs{\alpha}^2} ,
\end{equation}
where $\alpha$ and $\operatorname{Re}(\alpha)$ is the real part of $\alpha$. The average associated with such a factor has, in the sense of the global refinement algorithm,  the following weights
\begin{equation} \label{eqn:weights_irreducible_term}
 w_1 = \frac{1}{1+2  \operatorname{Re}(\alpha) + |\alpha|^2} , \quad w_2 = \frac{2  \operatorname{Re}(\alpha)}{1+2  \operatorname{Re}(\alpha) + \abs{\alpha}^2} , \quad w_3 = \frac{\abs{\alpha}^2}{1+2  \operatorname{Re}(\alpha) + \abs{\alpha}^2} . 
\end{equation}
Note that $w_1+w_2+w_3 = 1$. Instead of \eqref{eqn:symbol_linear_factorization} we have in this case the factorization
\begin{equation} \label{eqn:symbol_quadratic_factorization}
a(z) = z^{-s}(1+z) \left( \prod_{i=1}^{m_1} \frac{1+\alpha_iz}{1+\alpha_i}   \right) \left(  \prod_{i=m_1+1}^{m_1+m_2}  \frac{1+2\operatorname{Re}(\alpha_i)z+\abs{\alpha_i}^2 z^2 }{1+2\operatorname{Re}(\alpha_i)+\abs{\alpha_i}^2}  \right) .
\end{equation}

\begin{lemma} \label{lemma:positive_denominator}
For any complex $\alpha$, $\alpha \not \in \mathbb{R}$
\begin{equation} \label{eqn:positive_denominator}
1+2\operatorname{Re}(\alpha)+\abs{\alpha}^2 >0 .
\end{equation}
\end{lemma}
\begin{proof}
When $\operatorname{Re}(\alpha)\ge 0$, \eqref{eqn:positive_denominator} holds clearly, while if $\operatorname{Re}(\alpha)<0 $ and $\alpha$ is not real, then $-\operatorname{Re}(\alpha) < \abs{\alpha} $, and 
\[  1+2  \operatorname{Re}(\alpha) + \abs{\alpha}^2 >  1-2\abs{\alpha}+ \abs{\alpha}^2 = (1-\abs{\alpha})^2 \ge 0 . \]
\end{proof}

From Lemma \ref{lemma:positive_denominator} and \eqref{eqn:weights_irreducible_term} we conclude that $w_1$ and $w_3$ are always positive.

\subsection{The general algorithm of global refinement}

For an irreducible quadratic factor in \eqref{eqn:symbol_quadratic_factorization} one is required to average $3$ points on the manifold at once. Motivated by the pyramid averaging of Section \ref{subsec:local_interpretation}, we define such an average and term it a \textit{three pyramid}. 
\begin{definition} \label{def:three_pyramid}
For three points $p_1,p_2,p_3$ with corresponding weights $w_1,w_2,w_3$, the ``three pyramid" is 
\[  \mathcal{P}\left( (p_1,p_2,p_3),(w_1,w_2,w_3) \right) = M_r \left( M_{t_2}(p_3,p_2),M_{t_1}(p_2,p_1), \right) , \]
where the following constraints must hold
\begin{enumerate}
\item
$t_1 r = w_1$.
\item
$(1-t_1)r + t_2(1-r) = w_2$.
\item
$(1-t_2)(1-r) = w_3$.
\end{enumerate}
\end{definition}
\begin{remark}
Two remarks on Definition \ref{def:three_pyramid}:
\begin{enumerate}
\item
For numbers $f_1,f_2,f_3$ the three pyramid coincides with $w_1f_1+w_2f_2+w_3f_3$.
\item
The three constraints of Definition \ref{def:three_pyramid} are not independent. Since we always assume that $w_1+w_2+w_3 = 1$, the sum of the three constraints always holds.
\end{enumerate}
\end{remark}

The global refinement of Algorithm \ref{alg:Global_refinement} uses uniform averaging in each level. The following lemma shows that this is not possible for symbols with complex roots.
\begin{lemma} \label{lemma:three_pyramid}
There is no three pyramid of Definition \ref{def:three_pyramid} for the weights \eqref{eqn:weights_irreducible_term} with $t_1=t_2$. However, such a three pyramid exists with $t_1 > t_2$.
\end{lemma}

\begin{proof}
For the first claim of the lemma, we rewrite the constraints of Definition \ref{def:three_pyramid} with $t = t_1=t_2$. The case $t=0$ is impossible since by \eqref{eqn:weights_irreducible_term} and Lemma \ref{lemma:positive_denominator} $w_1>0$. Therefore, substitution of $r = \frac{w_1}{t}$ into the third constraint yields $t^2 + (w_3-w_1-1)t +w_1 = 0$, which has no real solution for the weights of \eqref{eqn:weights_irreducible_term}.

To prove the second claim, one can choose $r = \frac{1}{1+|\alpha|}$ for the weights in \eqref{eqn:weights_irreducible_term}. This yields a three pyramid with 
\begin{equation} \label{eqn:pyramid_paramters}
 t_1 = \frac{w_1}{r}=\frac{|\alpha|+1}{1+ 2 \operatorname{Re}(\alpha) + |\alpha|^2} , \quad  t_2 =  1- \frac{w_3}{1-r}  = \frac{1+2  \operatorname{Re}(\alpha)-\abs{\alpha}}{1+2  \operatorname{Re}(\alpha) + \abs{\alpha}^2}  .
\end{equation}
Note that for a non-real $\alpha$, $\abs{\alpha}> \abs{ \operatorname{Re}(\alpha)}$, and thus in view of Lemma \ref{lemma:positive_denominator}
\begin{equation} \label{eqn:difference_t1_t_2}
 t_1-t_2 =\frac{2(\abs{\alpha}-\operatorname{Re}(\alpha) )}{1+2  \operatorname{Re}(\alpha) + \abs{\alpha}^2}  >0 . 
\end{equation}
\end{proof}

The proof of Lemma \ref{lemma:three_pyramid} suggests a choice for the parameters of the three pyramid, for calculating the average of $3$ points at once. 
This choice, as is shown in Section \ref{subsec:optimal_choice}, is designed to minimize the bound on the distance between averages of two adjacent triplets of points, .

The adaptation of the global refinement algorithm corresponding to the symbol \eqref{eqn:symbol_quadratic_factorization}, based on geodesic averages and three pyramid averages, is summarized in Algorithm \ref{alg:Global_refinement_general_case}, which replaces Algorithm \ref{alg:Global_refinement} for symbols having complex roots.

\begin{algorithm}[ht]
\caption{Global refinement step -- the general case}
\label{alg:Global_refinement_general_case}
\begin{algorithmic}[1]
\REQUIRE  The coefficients $\alpha_1,\ldots,\alpha_{m_1+m_2}$ of the symbol \eqref{eqn:symbol_quadratic_factorization} and the value $s$. \\ Assume $\alpha_1$ is defined as in Theorem \ref{thm:positive_rots_convergence}. \\ The data to be refined by $\mathcal{S}$, $ \mathbf{p} = \{p_i\}_{i \in \mathbb{Z}}$.
\ENSURE The refined data $\mathcal{S} \left( \mathbf{p} \right) $.
\STATE  $q_{2i,0} \gets p_{i}$    \label{algG:init1}
\STATE  $q_{2i+1,0} \gets p_{i}$  \label{algG:init2}
\FOR[Go over each term corresponding to a real root of the symbol]{$j=1$ \TO $m_1$}    \label{algG:factor_loop}
\FOR{$i \in \mathbb{Z}$}  \label{algG:main_loop}
\STATE  $q_{i,j} \gets M_\frac{\alpha_j}{1+\alpha_j}(q_{i,j-1},q_{i+1,j-1}) $ \label{algG:average_loop}
\ENDFOR     \label{algG:data_const2}
\ENDFOR
\FOR[Go over each term corresponding to a complex root of the symbol]{$j=m_1+1$ \TO $m_1+m_2$}   
\FOR{$i \in \mathbb{Z}$}  \label{algG:complex_loop}
\STATE $w_1 \gets \frac{1}{1+2  \operatorname{Re}(\alpha_j) + \abs{\alpha_j}^2} $  
\STATE $w_2 \gets \frac{2  \operatorname{Re}(\alpha_j)}{1+2  \operatorname{Re}(\alpha_j) + \abs{\alpha_j}^2}$ 
\STATE $w_3 \gets  \frac{\abs{\alpha_j}^2}{1+2  \operatorname{Re}(\alpha_j) + \abs{\alpha_j}^2} $ 
\STATE  $  q_{i,j} \gets \mathcal{P}\left( (q_{i,j-1},q_{i+1,j-1},q_{i+2,j-1}),(w_1,w_2,w_3) \right)$
\ENDFOR     
\ENDFOR
\FOR[A final shifting]{$i \in \mathbb{Z}$} 
\STATE   $\mathcal{S} \left( \mathbf{p} \right)_{i-s+1}  \gets q_{i,m} $. \label{algG:shifting_step}
\ENDFOR     
\RETURN $\mathcal{S} \left( \mathbf{p} \right)$
\end{algorithmic}
\end{algorithm}

\subsection{Optimal choice of parameters in the three pyramid} \label{subsec:optimal_choice}

To optimally bound the distance  
\begin{equation} \label{eqn:dist_between_three_pyramids}
 d(\mathcal{P}\left( (p_1,p_2,p_3),(w_1,w_2,w_3) \right),\mathcal{P}\left( (p_2,p_3,p_4),(w_1,w_2,w_3) \right) . 
\end{equation}
we start by setting $r \in (0,1)$. The reasons for this choice are presented in details in Appendix \ref{appendix:why_r_01}. For the other parameters, we first prove the following Lemma.

\begin{lemma} \label{lemma:three_pyramid_parameters_t1t2}
Consider the three pyramid of Definition \ref{def:three_pyramid} for the weights \eqref{eqn:weights_irreducible_term} with $r \in (0,1)$. Then, $t_1>t_2$.
\end{lemma}

\begin{proof}
By the constraints of Definition \ref{def:three_pyramid}, $f(r) = t_1 -t_2= \frac{w_1}{r} + \frac{w_3}{1-r} -1  $. We show that $\min_{r \in (0,1)} f(r) >0$. Indeed, $f^\prime(r) = \frac{-w_1}{r^2} + \frac{w_3}{(1-r)^2}$, which implies a single minimum point of $f(r)$ at $r^\ast = \frac{\sqrt{w_1}}{\sqrt{w_1}+\sqrt{w_3}}=\frac{1}{1+\abs{\alpha}}$. By \eqref{eqn:difference_t1_t_2} we have that $f(r^\ast)>0$, and since $r^\ast$ is a minimum point, $f(r) \ge f(r^\ast)>0$. 
\end{proof}

\begin{theorem} \label{thm:three_pyramid_distance_bound}
Consider the three pyramid of Definition \ref{def:three_pyramid} with the weights \eqref{eqn:weights_irreducible_term} and $r \in (0,1)$. Then, for $p_1,p_2,p_3,p_4$ with $\delta(\mathbf{p}) = \max_{1 \le i \le 3} d(p_{i+1},p_i)$,
\begin{equation} \label{eqn:expand_three_pyramid}
d(\mathcal{P}\left( (p_1,p_2,p_3),(w_1,w_2,w_3) \right),\mathcal{P}\left( (p_2,p_3,p_4),(w_1,w_2,w_3) \right) ) \le \left( 2(t_1-t_2)+1  \right) \delta(\mathbf{p}) . 
\end{equation}
\end{theorem}
 
\begin{proof}
Figure \ref{fig:proof} accompanies the proof. There $M_1$ and $M_2$ correspond to $M_{t_1}(p_2,p_1)$ and $M_{t_2}(p_3,p_2)$ respectively, while $\overline{M_1}$ and $\overline{M_2}$ correspond to $M_{t_1}(p_3,p_2)$ and $M_{t_2}(p_4,p_3)$ respectively. $P$ and $\overline{P}$ there correspond to $\mathcal{P}\left( (p_1,p_2,p_3),(w_1,w_2,w_3) \right) $ and $\mathcal{P}\left( (p_2,p_3,p_4),(w_1,w_2,w_3) \right)$ respectively.

 We first apply the metric property \eqref{eqn:metric_property} and the triangle inequality to get (see the schematic illustration in Figure \ref{fig:eqn22})
\begin{equation} \label{eqn:pyramid_bound_proof1}
\begin{array}{r@{}l}
 d\left( M_{t_2}(p_3,p_2),M_{t_1}(p_2,p_1), \right) &{} \le  d(M_{t_2}(p_3,p_2),p_2)  + d(p_2,M_{t_1}(p_2,p_1)) \\ 
 &{}=  (1-t_2)d(p_2,p_3) + t_1 d(p_1,p_2) .  
\end{array}
\end{equation}
Note that $t_1 = \frac{w_1}{r} >0$ and that $1-t_2 = \frac{w_3}{1-r} >0$. Similarly we get
\[  d(M_{t_1}(p_3,p_2),p_2) = (1-t_1)d(p_2,p_3)  , \]
and since $1-t_2 > 1-t_1$ by Lemma \ref{lemma:three_pyramid_parameters_t1t2}, we conclude that $M_{t_1}(p_3,p_2)$ is closer to $p_2$ than $ M_{t_2}(p_3,p_2)$ (see Figure \ref{fig:eqn23}). Observing that these two averages lie on the geodesic curve connecting $p_2$ and $p_3$, we conclude that
\begin{equation} \label{eqn:pyramid_bound_proof2}
 d(M_{t_1}(p_3,p_2),M_{t_2}(p_3,p_2))  = \left((1-t_2)-(1-t_1)\right) d(p_2,p_3) = (t_1-t_2)d(p_2,p_3) . 
\end{equation}

To prove \eqref{eqn:expand_three_pyramid} we sum the following three bounds, on the lengths of the three parts of the path connecting $P$ to $\overline{P}$ via $\mathit{M_2}$ and $\overline{\mathit{M_1}}$ in Figure \ref{fig:overview},
\begin{eqnarray*}
d( \mathcal{P}\left( (p_1,p_2,p_3),(w_1,w_2,w_3) \right) ,M_{t_2}(p_3,p_2)) &\le & (1-r)(t_1+(1-t_2))\delta(\mathbf{p}) , \\
d(M_{t_2}(p_3,p_2),M_{t_1}(p_3,p_2)) &\le  & (t_1-t_2)\delta(\mathbf{p})  ,\\
d( M_{t_1}(p_3,p_2), \mathcal{P}\left( (p_2,p_3,p_4),(w_1,w_2,w_3) \right) ) &\le &  r(t_1+(1-t_2)) \delta(\mathbf{p}) .
\end{eqnarray*}
The first and third bounds are obtained from Definition \ref{def:three_pyramid} by \eqref{eqn:metric_property} and \eqref{eqn:pyramid_bound_proof1}, the second bound is \eqref{eqn:pyramid_bound_proof2}.

\begin{figure}
\centering    
\subfloat[$\mathit{M_1} = M_{t_1}(p_2,p_1)$ and $\mathit{M_2} = M_{t_2}(p_3,p_2)$ ]{\label{fig:eqn22}			   \includegraphics[width=0.5\textwidth]{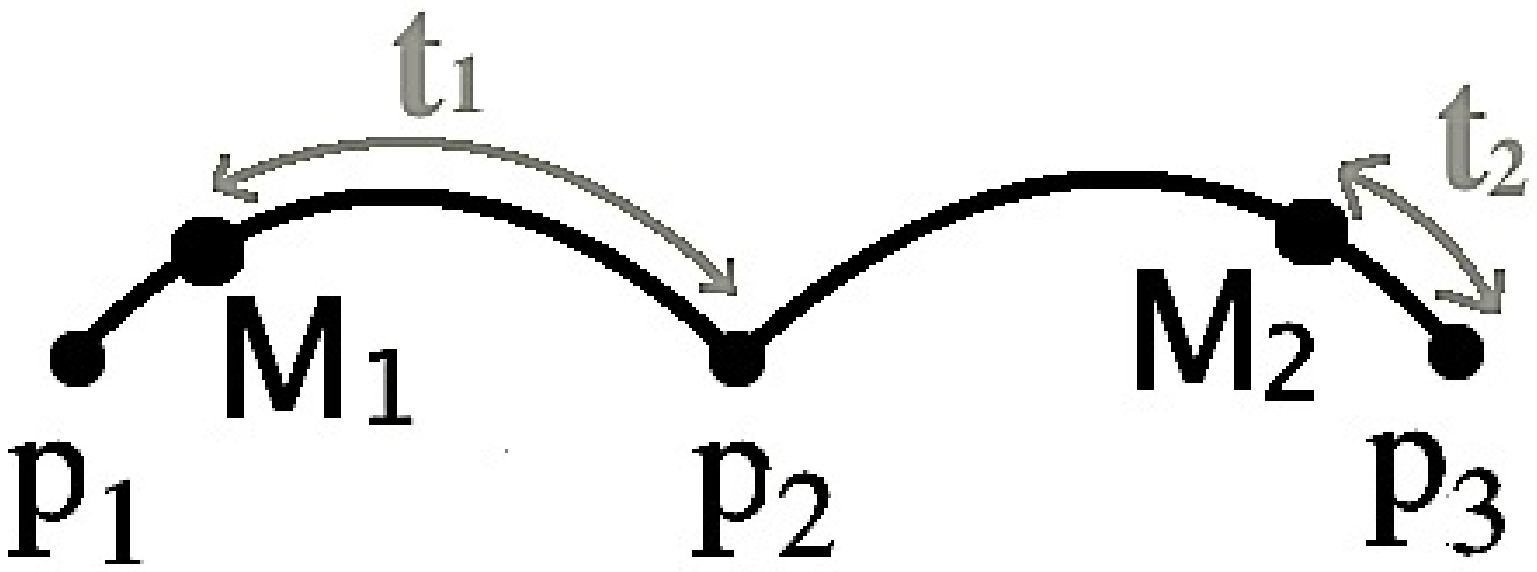}} \quad
\subfloat[ $\overline{\mathit{M_1}} = M_{t_1}(p_3,p_2)$ and $\mathit{M_2} = M_{t_2}(p_3,p_2)$]{\label{fig:eqn23} \includegraphics[width=0.4\textwidth]{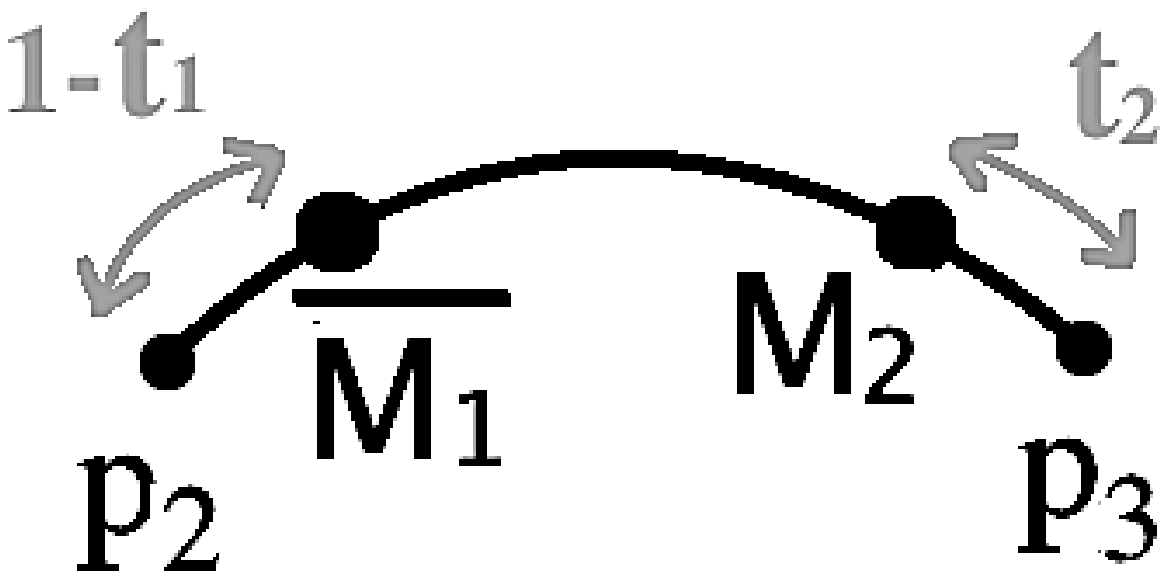}} \\
\subfloat[$P = \mathcal{P}\left( (p_1,p_2,p_3),(w_1,w_2,w_3) \right)$ and ${\overline{P}= \mathcal{P}\left( (p_2,p_3,p_4),(w_1,w_2,w_3) \right)}$. ]{\label{fig:overview}      \includegraphics[width=0.65\textwidth]{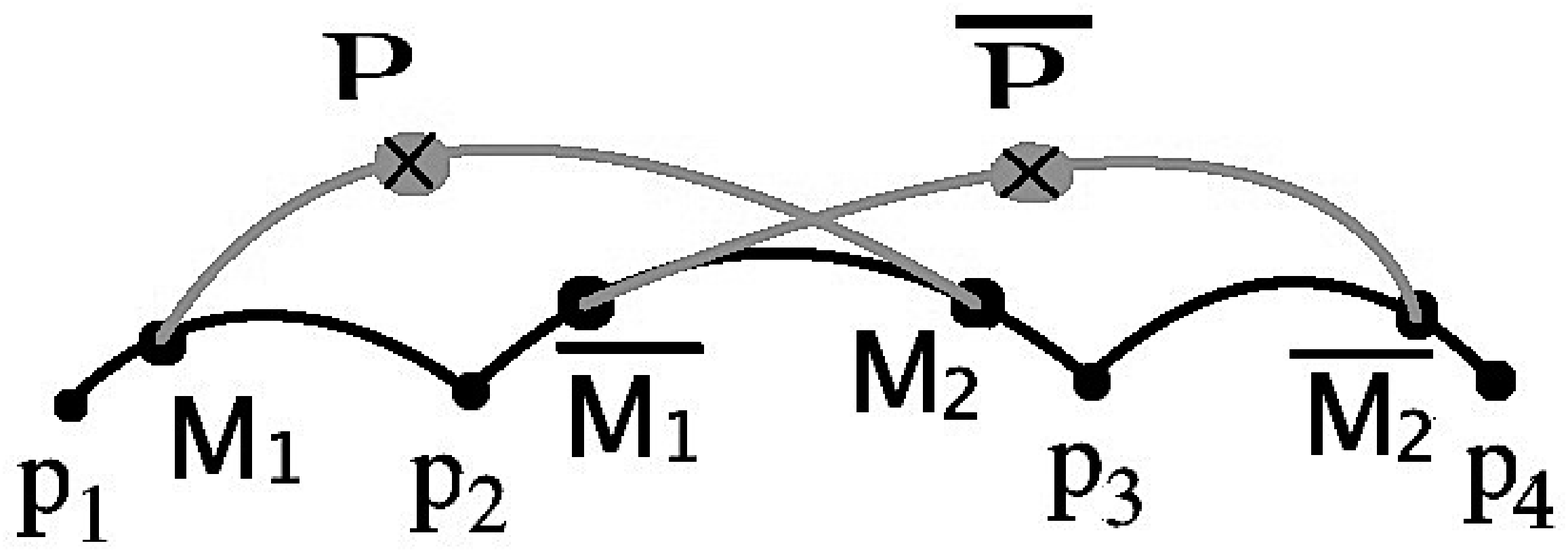}}
\caption{Illustration for the proof of Theorem \ref{thm:three_pyramid_distance_bound}. The curved lines (arcs) symbolically represent geodesic curves connecting two points. The bright arrows in upper figures describe the relative distances compared to each of the corresponding geodesics. }
\label{fig:proof}
\end{figure}

\end{proof}
\begin{remark} \label{rem:pyramid_parameters}
Two important conclusions, related to the parameters of the three pyramid:
\begin{enumerate}
\item
Theorem \ref{thm:three_pyramid_distance_bound} implies that in order to reduce the expansion factor in \eqref{eqn:expand_three_pyramid} corresponding to a three pyramid the function $f(r)$, from the proof of Lemma \ref{lemma:three_pyramid_parameters_t1t2} has to be minimized. Thus, the parameters $t_1$ and $t_2$ of \eqref{eqn:pyramid_paramters} and $r = \frac{1}{1+\abs{\alpha}}$ are preferred. 
\item
For the parameters in the first part of the remark, we deduce from Lemma \ref{lemma:three_pyramid_parameters_t1t2} that the bound in \eqref{eqn:expand_three_pyramid} is bigger than one. This means that the bound $\delta(\mathbf{p})$ on the distances between adjacent points is not preserved after applying the three pyramid.
\end{enumerate}
\end{remark}

Note that in the linear case, any averaging step corresponding to a complex root does not expand the distance between consecutive points as long as the weights \eqref{eqn:weights_irreducible_term} are positive, that is the real part of $\alpha$ is positive.

\subsection{Analysis of convergence}

First, we consider the case of symbols of the form \eqref{eqn:symbol_quadratic_factorization} having several complex roots and then discuss in detail the case of a single complex root.

In case of positive roots, which is analysed in Theorem \ref{thm:positive_rots_convergence}, we show an initial contractivity factor induced by $\alpha_1>0$, associated with the negative root, followed by a series of expanding factors $\xi(\alpha_i)$ for $\alpha_i<0$, associated with the positive roots. Equipped with Theorem \ref{thm:three_pyramid_distance_bound}, the analysis of the convergence of the schemes adapted by Algorithm \ref{alg:Global_refinement_general_case} is essentially the same.

\begin{corollary} \label{cor:multiply_complex_roots}
Let $\mathcal{S}$ be a linear subdivision scheme with symbol $a(z)$ of the form \eqref{eqn:symbol_quadratic_factorization}, with $m_1,m_2 \ge 1$ and $\max_{1 \le i \le m_1} \alpha_i >0 $. Define
\begin{equation*}   
 \mu_1 = \min_{\substack{\alpha_i>0 \\ 
     i \in \{1,\ldots,m_1 \} } }  \max\{ \frac{1}{1+\alpha_i},\frac{\alpha_i}{1+\alpha_i}  \} , 
\end{equation*}
and renumerate the linear factors in \eqref{eqn:symbol_quadratic_factorization} such that $\mu_1$ is attained at $\alpha_1$. If
\begin{equation} \label{eqn:complex_roots_contractive_factor}
\mu = \mu_1 \prod_{i=2}^{m_1+m_2} \xi(\alpha_i) <1 , 
\end{equation}
where 
\[ \xi(\alpha_i) = 
\begin{cases} 
1,                                                     & 0<\alpha_i      , \\
1+2 \abs{\frac{\alpha_i}{1+\alpha_i}},  & -1<\alpha_i<0 , \\
1+2 \abs{\frac{1}{1+\alpha_i}},          & \alpha_i<-1      , \\
1+2 \left( \frac{2(\abs{\alpha_i}-\operatorname{Re}(\alpha_i) )}{1+2  \operatorname{Re}(\alpha_i) + \abs{\alpha_i}^2}  \right)  & \alpha_i \not \in \mathbb{R} .
\end{cases}
\]
then the adapted scheme based on global refinement has a contractivity factor $\mu$, and it converges from any initial admissible data on the manifold.
\end{corollary}
The proof is in the spirit of the proof of Theorem \ref{thm:positive_rots_convergence} and is based on Theorem \ref{thm:three_pyramid_distance_bound} and the choice \eqref{eqn:pyramid_paramters} of the parameters. Note that similar arguments (as mentioned in the proof of Theorem \ref{thm:positive_rots_convergence}) also confirms that the proof of Theorem \ref{thm:DS_condition_negative_roots} holds in the case of complex roots, with $K_{j+1} = \frac{3}{2} +  K_j$. Thus, the full proof is omitted.

A similar sufficient condition for the convergence of the adapted scheme with refinement step as in Algorithm \ref{alg:Global_refinement_general_case} is
\begin{corollary} \label{cor:uniform_bound_for_expanding}
In the notation of Corollary \ref{cor:multiply_complex_roots}, if
\[ 1+2 \left( \frac{2(\abs{\alpha_i}-\operatorname{Re}(\alpha_i) )}{1+2  \operatorname{Re}(\alpha_i) + \abs{\alpha_i}^2}  \right) < \left( \frac{1}{\mu_1}\prod_{j=2}^{m_1} \xi(\alpha_j)   \right)^\frac{1}{m_2} , \quad  i = m_1+1,m_1+2, \ldots, m_1+m_2 , \]
then, the adapted scheme is convergent for all admissible input data. 
\end{corollary}

We provide an additional perspective to the above analysis by assuming only one irreducible quadratic factor with all real linear factors corresponding to negative roots. In such a scenario, we can describe exactly the domain in the complex plane from which a single complex $\alpha$ leads to a convergent adapted scheme. This can be extended to several complex roots using the same approach as in Corollary \ref{cor:uniform_bound_for_expanding}.
\begin{theorem} \label{thm:single_complex_root}
Let $\mathcal{S}$ be a linear subdivision scheme, with a symbol of the form \eqref{eqn:symbol_quadratic_factorization}, adapted by Algorithm \ref{alg:Global_refinement_general_case} such that $m_1\ge 1$, $m_2=1$ and $\alpha_i>0$, $1 \le i \le m_1$. Then, the adapted scheme converges from all admissible input data, whenever $\alpha_{m_1+1}$ is outside the domain $\Omega$ given by
\[  \Omega =\left\lbrace  re^{i\phi}   \mid  \rho_1(\phi) \le  r \le \rho_2(\phi) , \quad   \upsilon < \phi < 2\pi-\upsilon  \right\rbrace  \cup  \left\lbrace e^{i\upsilon},e^{-i\upsilon} \right\rbrace. \]
Here $0<\upsilon = \arccos(\frac{3\mu_1-1}{1+\mu_1}) < \arccos(\frac{1}{3})$, and the curves $\rho_1$ and $\rho_2$ are 
\[ \rho_{1,2} (\phi) = \frac{-(1+\mu_1)}{1-\mu_1}\cos(\phi) + \frac{2\mu_1}{1-\mu_1} \mp \sqrt{\left( \frac{-(1+\mu_1)}{1-\mu_1}\cos(\phi) + \frac{2\mu_1}{1-\mu_1} \right)^2-1} ,\]
where $\mu_1$ is the initial contractivity factor $ \mu_1 =  \max_{1\le i \le m_1}\{ \frac{1}{1+\alpha_i},\frac{\alpha_i}{1+\alpha_i} \}$.
\end{theorem}
The proof is given in Appendix \ref{appendix:proof_single_root}.

First, note that $\Omega$ is symmetric relative to the real axis. To further illustrate $\Omega$ and the complemented domain of convergence $\mathbb{C} \backslash \Omega$ we refer the reader to Figure \ref{fig:omega_domain}, where the domain of convergence for a single irreducible factor and an initial contractivity factor $\mu=\frac{1}{2}$ is presented. This value of $\mu$ implies that $-1$ has multiplicity as a root of the symbol, which is typical to $C^1$ schemes. The convergence domain includes all the complex plane but $\Omega$, and one can clearly notice the domain $\abs{\operatorname{arg}(\alpha)}< \upsilon$ around the positive real axis (between the dashed lines), where there is no restriction on the modulus of the complex $\alpha_{m_1+1}$.

\begin{figure}[ht]
 \centering
 \includegraphics[width=.8\textwidth]{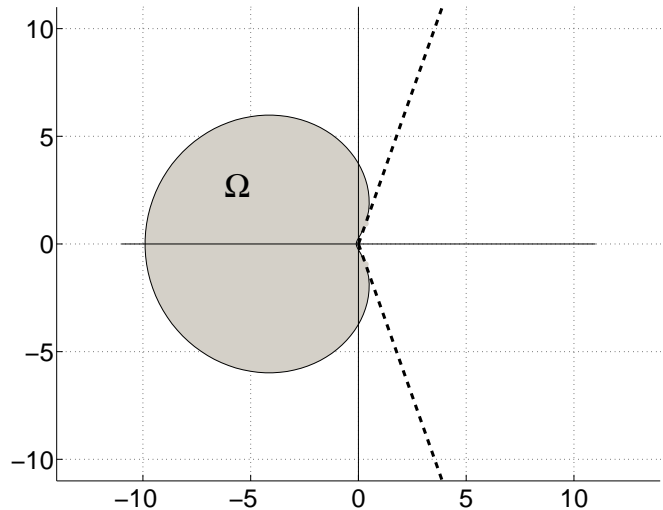} 
 \caption{The domain of convergence $\mathbb{C} \backslash \Omega$ for the case of a single irreducible real quadratic factor and an initial contractivity factor $\mu=\frac{1}{2}$. The dashed lines are $\operatorname{arg}(\alpha) = \pm \upsilon$. }
 \label{fig:omega_domain}
\end{figure}

\begin{remark}
An interesting class of manifolds is the Hadamard manifolds which are globally non-positively curved metric spaces, see e.g., \cite{bridson1999metric}. 
On a Hadamard manifold, any two points $p_1$ and $p_2$, and their connecting geodesic $M_t(p_1,p_2)$, $t \in [0,1]$ satisfy for any point $q$ on the manifold
\[ d^2(q,M_t(p_1,p_2)) \le (1-t)d^2(q,p_1) + td^2(q,p_2) - t(1-t)d^2(p_1,p_2) . \]
Such manifolds are also called (global) CAT(0)-spaces and NPC spaces. Contrary to general manifolds, where geodesics are merely locally the shortest path, in Hadamard manifolds the geodesics are unique and global. 

In Hadamard manifolds, one can establish superior bounds on distances as \eqref{eqn:dist_between_three_pyramids}, between averages of more than two points. For example, in \cite{lim2013approximations}, a class of ``weak contractivity" averages is introduced. Distances of the form \eqref{eqn:dist_between_three_pyramids}, based on such averages, are bounded by $\delta(\mathbf{p})$. Thus, in Hadamard manifolds, irreducible quadratic factors, can be replaced by weak contractivity averages in the global refinement algorithm. With this modification, the contractivity factor is independent of the number of such factors and the convergence of schemes based on the global refinement is guaranteed for any symbol with all roots having negative real parts and at least one negative root in addition to $-1$. Note that all such symbols have positive coefficients.
\end{remark}


\bibliography{manifold_bib}
\bibliographystyle{plain} 


\appendix  \label{appendix}

\section{Supplements for Section \ref{sec:global_refinements_complex_root} } 

\subsection{ Why to choose $0<r<1$ in a three pyramid?} \label{appendix:why_r_01}

The main argument for choosing $0<r<1$ in a three pyramid is to avoid the use of high extrapolation values in the averages of the three pyramid. Namely, we wish to minimize the use of averaging parameters that are much bigger than one or much smaller than zero. 

To simplify the discussion, we focus on the left part of the complex domain, namely consider complex roots with negative real parts, that is $\alpha_j$ such that $\operatorname{Re}(\alpha_j) > 0$, $j=m_1+1,...,m_1+m_2$ in \eqref{eqn:symbol_quadratic_factorization}.
\begin{proposition} \label{prop:three_pyramid_parmeters}
Consider a three pyramid, corresponding to a complex $\alpha_j$ with $\operatorname{Re}(\alpha_j) > 0$, with the parameters $t_1,t_2$ given by \eqref{eqn:pyramid_paramters} for $r = \frac{1}{1+\abs{\alpha_j}} \in (0,1)$. Then, 
\begin{enumerate}
\item
$t_1>0$ and $t_2<1$. Moreover, at least one of $t_1,t_2$ is in $(0,1)$.
\item
In case $t_1 \not \in (0,1)$, then $0<t_1 \le \frac{1+\sqrt{2}}{2} \approx 1.207 $.
\item
In case $t_2 \not \in (0,1)$, then $ -0.207 \approx \frac{1-\sqrt{2}}{2}  \le t_2 <1$.
\end{enumerate}
\end{proposition}
\begin{proof}
For the first claim, we note that by \eqref{eqn:pyramid_paramters} $t_1>0$ (since $r>0$ and $w_1>0$), and $t_2<1$. For the rest of the first claim, we consider two cases. When $\abs{\alpha_j} \ge 1$ 
\[ 1+2\operatorname{Re}(\alpha_j)+\abs{\alpha_j}^2>1+\abs{\alpha_j}^2 \ge  1+ \abs{\alpha_j} ,\] 
and $t_1<1$ by \eqref{eqn:pyramid_paramters}. On the other hand, when $\abs{\alpha_j}<1$, we have that 
\[  1+2\operatorname{Re}(\alpha_j)-\abs{\alpha_j} >  1+2\operatorname{Re}(\alpha_j)-1 >0   , \] 
and it follows from \eqref{eqn:pyramid_paramters} that $t_2>0$.

For the second claim, denote $\alpha_j=\rho e^{i \theta}$, which leads to $t_1 = g(\rho,\theta) = \frac{1+\rho}{1+2\rho\cos(\theta)+\rho^2}$. Then, a standard analysis using differentiation shows no extreme points for $g$ inside the domain $\operatorname{Re}(\alpha_j) > 0$. On the boundary of this half plane, that is $\theta = \pm \frac{\pi}{2}$, there are two maximum points at $\rho =\sqrt{2}-1$, yielding the bound on $t_1$. The third claim is proved similarly. One finds that $1-t_2$ has the same maximal values as $t_1$.
\end{proof}

Proposition \ref{prop:three_pyramid_parmeters} shows that choosing the parameters \eqref{eqn:pyramid_paramters} with $r = \frac{1}{1+\abs{\alpha_j}} \in (0,1)$ guarantees at most one extrapolating average, with a weight just slightly outside $(0,1)$, namely in $(-0.207,1.207)$. For $r \not \in (0,1)$ this is not the case.

Recall the general expressions of the parameters $t_1 = \frac{w_1}{r}$ and $1-t_2 = \frac{w_3}{1-r}$. These expressions reveal that if $r \not \in (0,1)$ both $t_1$ and $t_2$ cannot be in $ (0,1)$. To get $t_1 \in (0,1)$ and $t_2$ bigger than $1$ but close to it, $r$ has to be sufficiently large, while to get $t_2 \in (0,1)$ and $t_1<0$ but close to $0$, $r$ must be negative with $\abs{r}$ sufficiently large. Moreover, if $r \not \in (0,1)$ but close to $(0,1)$ either $t_1$ or $t_2$ become unbounded. To further demonstrate this, we present a simple example.
\begin{example}
We illustrate the extreme extrapolation values required for the case of $r \not\in (0,1)$ by calculating the parameters of the three pyramid for the special case $\alpha_j = 1+\frac{1}{2}i \in \mathbb{C}$. Note that for this root, when using $r = \frac{1}{1+\abs{\alpha_j}} = 0.4721$, the corresponding parameters are $t_1 \approx 0.4984$, $t_2 \approx 0.4428$. Furthermore, for the case of a single quadratic factor, as done in Theorem \ref{thm:single_complex_root}, the scheme has a contractivity factor for any $\mu_1<0.9$.

On the other hand, allowing small $r$ values of extrapolation results in high, undesired extrapolation values of $t_2$ (when $r>1$) or of $t_1$ (when $r<0$). This is demonstrate in Tables \ref{tab:chapter4:1a} and \ref{tab:chapter4:1b}, where as $r$ gets closer to $(0,1)$, either $t_1$ or $t_2$ get further away from $(0,1)$.

\begin{table}[!ht]
  \centering
  \subfloat[Case of $r>1$  \label{tab:chapter4:1a}]{
  \hspace{.5cm}%
    \centering
  \begin{tabular}{c|c|c }
    $r$ & $t_1$ & $t_2$ \\\hline\hline    
    1.5 &    0.1569  &  1.5882 \\ \hline
    1.4 &   0.1681 &  1.7353 \\ \hline
    1.3 &    0.1810  & 1.9804 \\ \hline
    1.2 &    0.1961 &  2.4706 \\ \hline
    1.1 & 0.2139 &  3.9412 
  \end{tabular}
  }
  \qquad
  \subfloat[Case of $r<0$ \label{tab:chapter4:1b}]{
  \hspace{.5cm}%
    \centering
  \begin{tabular}{c|c|c }
    $r$ & $t_1$ & $t_2$ \\\hline\hline
   -0.5 &   -0.4706   &  0.8039  \\ \hline
   -0.4 &   -0.5882  &  0.7899 \\ \hline
   -0.3 &   -0.7843   & 0.7738 \\ \hline
   -0.2 &   -1.1765  &  0.7549 \\ \hline
   -0.1 &   -2.3529 &   0.7326
  \end{tabular}
  }
  \caption{The parameters of the three pyramid for $\alpha_j = 1+\frac{1}{2} i$}
  \label{tab:extrapolation_values}
\end{table}

Note that another outcome of high extrapolation values of $r$ is that the convergence domain, $\mathbb{C} \backslash \Omega$ of Theorem \ref{thm:single_complex_root}, becomes more restrictive than the one obtained for $r\in(0,1)$. The proof for this claim can be easily understood but involves many technical details and thus is omitted. 
\end{example}

\subsection{Proof of Theorem \ref{thm:single_complex_root} } \label{appendix:proof_single_root}

As in Corollary \ref{cor:multiply_complex_roots}, it is sufficient to ensure a contractivity factor. Recall that $\alpha_i >0$, $i=1,\ldots,m_1$. Accordingly, we have that by reaching Line \ref{algG:complex_loop} of Algorithm \ref{alg:Global_refinement_general_case} we retain the bound $\mu_1 \delta(\mathbf{p})$ on the distance between adjacent points. Using Theorem \ref{thm:three_pyramid_distance_bound} we get that a sufficient condition for having a contractivity factor is
\[ 2(t_1-t_2) +1 < \frac{1}{\mu_1}  .\]
Substituting \eqref{eqn:difference_t1_t_2} and $\alpha_{m_1+1} = \rho e^{i\theta}$ we get the sufficient condition for contractivity
\begin{equation} \label{eqn:parabola_to_analyse}
 \rho^2-2(\gamma (1-\cos(\theta))-\cos(\theta)) \rho +1 >0  , 
\end{equation}
with $\gamma = \frac{2\mu_1}{1-\mu_1}$. 

For a fixed $\theta$, consider the left-hand side of \eqref{eqn:parabola_to_analyse} as a parabola in $\rho$ and denote it by $h(\rho)$. Then, $h^\prime(\rho) = 2\rho- 2(\gamma (1-\cos(\theta))-\cos(\theta)) $. The derivative implies that the minimum, as a function of $\rho$, is obtained at 
\[ \rho^\ast =  \gamma (1-\cos(\theta))-\cos(\theta), \]
for a fixed $\theta$.

We divide the analysis into two different cases and start with the case that \eqref{eqn:parabola_to_analyse} holds for any $\rho>0$. Since the parabola $h(\rho)$ has a minimum and satisfies $h(0)=1$, there are two scenarios: the first is $\rho^\ast<0$ and the second is $\rho^\ast \ge 0$ and $h(\rho^\ast) = 1-(\rho^\ast)^2>0$, namely $ 0 \le \rho^\ast<1$.  Therefore, a combined condition for the two scenarios is simply $\rho^\ast<1$, or $\cos(\theta)>\frac{3\mu_1-1}{1+\mu_1}$. Thus, the argument of the cosine must satisfies $\theta \in (-\upsilon,\upsilon)$, where $\upsilon= \arccos(\frac{3\mu_1-1}{1+\mu_1})$, with $\frac{1}{3} \le |\frac{3\mu_1-1}{1+\mu_1}|<1$, since $\frac{1}{2} \le \mu_1<1$. This is the domain where we have a contractivity factor for all $\rho$.

The second case is when the parabola $h(\rho)$ has two positive roots. In this case we have a non-negative discriminant, that is $(\gamma (1-\cos(\theta))-\cos(\theta))^2 -1 \ge 0$, or equivalently $\gamma (1-\cos(\theta))-\cos(\theta) \ge 1$ (the case $\gamma (1-\cos(\theta))-\cos(\theta) \le -1$ was already treated above, since in this case $\rho^\ast<0$). The equality corresponds to the case $\theta =  \pm \upsilon$ namely to a vanishing discriminant. In this case $h(\rho) = (\rho-1)^2$ and \eqref{eqn:parabola_to_analyse} holds for $\rho \neq 1$. Otherwise, we have contractivity when $\rho$ is bigger than the large root or smaller than the small root of $h(\rho)$. The roots are curves, parameterized by $\phi \in (\upsilon,2\pi-\upsilon)$, as appears in the statement of the theorem.


\end{document}